\documentclass{article}

\usepackage{epigraph}
\usepackage[utf8]{inputenc}
\usepackage[english]{babel}
\usepackage{amsfonts}
\usepackage{amssymb}
\usepackage{amsmath}
\usepackage{xypic}
\usepackage[all]{xy}
\usepackage{mathrsfs}
\usepackage{latexsym}
\usepackage{amscd}

\usepackage[mathscr]{eucal}

\usepackage{amsthm}

\usepackage{setspace}                    

\usepackage{tocbibind}

\usepackage{hyperref}

\DeclareMathOperator{\codim} {codim}

\newtheorem{prop}{Proposition}[section]
\newtheorem{proposition}[prop]{Proposition}

\newtheorem{theorem}[prop]{Theorem}
\newtheorem{lem}[prop]{Lemma}
\newtheorem{lemma}[prop]{Lemma}

\theoremstyle{definition}

\newtheorem{definition}[prop]{Definition}

\theoremstyle{remark}

\newtheorem{remark}[prop]{Remark}
\newtheorem{notation}[prop]{Notation}
\newtheorem{example}[prop]{Example}
\newtheorem{qst}[prop]{Question}
\newtheorem{question}[prop]{Question}

\newcommand{\F}{\mathbb{F}}
\newcommand{\G}{\mathbb{G}}

\newcommand{\Sym}{\mathrm{Sym}}

\usepackage{tikz}
\usetikzlibrary{arrows}

\newcommand{\pp}{\mathbb{P}}

\newcommand{\oo}{\mathcal{O}}

\newcommand{\rk}{\mathrm{rk}}

\title{Iitaka fibrations for vector bundles}
\author{Ernesto C. Mistretta, Stefano Urbinati}
\date{}

\begin{document}

\maketitle

\epigraph{\itshape Think you're escaping and run into yourself. Longest way round is the shortest way home.}{---James Joyce, \textit{Uulysses - Nausicaa}}

\begin{abstract}
A  vector bundle on a smooth projective variety, if it is generically generated by global sections,
yields a rational map to a Grassmannian,
called Kodaira map.
We investigate the asymptotic behaviour of the Kodaira maps for the symmetric powers of a vector bundle,
and we show that these maps stabilize to a map dominating all of them, 
as it happens for a line bundle via the Iitka fibration.
Through this Iitaka-type construction, applied to the cotangent bundle, 
we give a new characterization of Abelian varieties.

\end{abstract}

\section{Introduction}

The aim of this work is to  to extend the notions of stable base loci, Kodaira maps and Iitaka fibrations to vector bundles,
in the same framework as \cite{4bros}.

This work is based on a key concept in birational geometry, i.e. that, 
given an algebraic variety $X$, the global sections of a line bundle on $X$ naturally induce
a (rational) map in some projective space.
It is  well-known, for example, that sections of some powers of an ample line bundle give
an embedding in a projective space. 
For more general line bundles this does not hold anymore.
There is, however, a classical and well developed theory of Iitaka fibrations 
describing all the different possible outcomes, see \cite{PAG}.

In the construction of the Iitaka fibration there are three  main steps. 
The first one is to decide if and where these maps are well defined (as regular morphisms), 
by studying the stable base locus of the given line bundle. 
Second, the fibration needs to behave nicely in an asympthotic way, 
hence the maps, that we will call \emph{Kodaira maps}, 
induced by different powers of the same line bundle have to be related. 
Third and last, the images shall stabilize, be controlled by a morphism and give information about the original variety.
All of this is known in the case of line bundles.

The main difference respect to the case of line bundles is that
global sections of vector bundles naturally induce maps into Grassmannians rather then projective spaces. 
Whereas the main technical issue is that the rank of the bundle grows as we take powers of it,
we are still able to prove that it is possible to give an Iitaka type construction in this more general case.
The results are similar to the ones achieved for line bundles, up to a finite map. 
In particular we do not always obtain a  fibration, a surjective morphism with connected fibres.

The paper is organized as follows. 

In section \ref{baseloci} we recall the main notions and results on asymptotic base loci for vector bundles introduced in \cite{4bros},
fixing some of the definitions. Furthermore, we correct the original definition of augmented and restricted base loci 
given in \cite{4bros}, as there were some inaccuracies.

In section \ref{semiample} we focus on the definition os semiampleness for vector bundles pointing out that the two characterizations,
often  considered equivalent in the folklore of vector bundles and positivity constructions, are actually not. 
In the same section, we provide an Iitaka construction for \emph{strongly semiample} vector bundles.

In the last section,  we give an Iitaka type statement for \emph{Asymptotically Generically Generated} vector bundles. 
We conclude the section with a characterization of Abelian varieties via  Iitaka-type invariants and base loci of the cotangent bundle.

\subsection{Acknowledgments} We would like to thank Francesco Esposito and Francesco Polizzi for extremely helpful and very pleasant conversations,
and we thank particularily Andreas H\"{o}ring for his kind advice and remarks.

\section{Base loci and projectivizations}
\label{baseloci}

In  \cite{4bros} the authors
established relations between positivity properties of vector bundles  on a projective variety $X$
and some corresponding properties of asymptotic base loci in $X$, in the same flavor of what happens for line bundles.

Because of some extra difficulties arising in the case of vector bundles, the 
various definitions of positivity (\emph{i.e.} bigness, semiampleness, etc.) 
appearing in the literature do differ according to different authors.

In this section we review the main properties of asymptotic base loci, their relations to positivity properties,
and some of the different definitions appearing in the literature, focusing on what will be needed in the Iitaka construction: semiampleness
and asymptotic generic generation.
We also take the chance to correct some definitions stated in \cite{4bros}  which were not very precise.

\begin{notation}
  Let $X$ be a normal projective variety over the complex numbers, and $E$ a vector
  bundle over $X$. 
  For a
  point $x\in X$, $E_x=E\otimes_{\oo_X}\oo_{X,x}$ denotes the stalk of $E$ at the
  point $x$, and $E(x)=E\otimes_{\oo_X} \kappa(x)$ denotes the fibre of $E$ at the point $x$, where $\kappa(x)$ is the residue
  field at $x$. 
  Clearly, $E(x)$ is a vector space of dimension $r=\rk E$.
\end{notation}

\begin{definition}

  We define the \emph{base locus} of $E$ (over $X$) as the subset
\[
\mathrm{Bs} (E) := \{ x \in X ~|~ H^0(X, E) \to E(x) ~\textrm{is not surjective} \}\ ,
\]
and the \emph{stable base locus} of $E$ (over $X$) as
\[
\mathbb{B} (E) := \bigcap_{m>0} \mathrm{Bs}(\Sym^m E)\ .
\]
\end{definition}

\begin{remark} The assertions below follow immediately from the definition:
  \begin{enumerate}
  \item $E$ is globally generated, \emph{i.e.} generated by its global sections, if and only if 
  $\mathrm{Bs} (E) = \emptyset$.
  \item As $\mathrm{Bs} (E) = {\rm Bs} (\mathrm{Im} (\bigwedge^{\rk E} H^0(X, E) \to
    H^0(X, \det E) ))$, these loci are closed subsets, and carry a natural scheme
    structure.
      \end{enumerate}
\end{remark}
%
%
%
%

\begin{definition}

  Let $r = p/q \in \mathbb{Q}^{>0}$ be a positive rational number, and $A$ a line
  bundle on $X$.  We will use the following notation:
  \begin{align*}
    \mathbb{B}(E + rA) &:=\bigcap_{m >0} \mathrm{Bs} (\Sym^{mq} E \otimes A^{mp}), \quad\text{and}\\
    \mathbb{B}(E - rA) &:= \bigcap_{m >0} \mathrm{Bs} (\Sym^{mq} E \otimes A^{-mp}) \ .
  \end{align*}

Let $A$ be an ample line bundle on $X$,
we define the \emph{augmented base locus} of $E$ as
\[
\mathbb{B}_+^A (E):= \bigcap_{r \in \mathbb{Q}^{>0}} \mathbb{B}(E - r A ) \ ,
\]
and the \emph{restricted base locus} of $E$ as
\[
\mathbb{B}_-^A (E):= \bigcup_{r \in \mathbb{Q}^{>0}} \mathbb{B}(E + r A )\ .
\]

\end{definition}

\begin{remark}

 Note that if $r=p'/q'$ is another representation of $r$ as a fraction, then
  $q'p=p'q$, hence
  $$
  \Sym^{q'q} E \otimes A^{q'p}\simeq
  \Sym^{q'q} E \otimes A^{p'q},
  $$
therefore, 
$\mathrm{Bs}(\Sym^{m q^{\prime} q} E \otimes A^{m q^{\prime} p})
  =\mathrm{Bs}(\Sym^{mq q^{\prime}} E \otimes A^{mq p^{\prime}})$ 
  and hence $\mathbb{B}(E + rA)$ is
  well-defined. A similar argument shows that $\mathbb{B}(E - rA)$ is well-defined as well.

\end{remark}

\begin{remark}

In \cite{4bros} the loci  $\mathbb{B}(E + rA)$  and
$\mathbb{B}(E - rA) $ were defined by 
$\mathbb{B}(E + rA) = \mathbb{B} (\Sym^{q} E \otimes A^{p})$
and $\mathbb{B}(E - rA) = \mathbb{B} (\Sym^{q} E \otimes A^{-p})$.
However, as 
$\Sym^m (\Sym^{q} E \otimes A^{p}) \neq \Sym^{mq} E \otimes A^{mp}$,
the loci defined here better suite our purposes.

\end{remark}

\begin{remark}

The definitions above yield the following properties:

\begin{enumerate}

\item The loci \( \mathbb{B}_+^A (E)\) and \( \mathbb{B}_-^A (E)\) do not depend on
  the choice of the ample line bundle $A$, so we can write
  \( \mathbb{B}_+ (E)\) and \( \mathbb{B}_- (E)\) for the augmented and restricted base locus of $E$, respectively.

\item For any $r_1 > r_2 > 0$ we have \( \mathbb{B} (E + r_1 A) \subseteq \mathbb{B}
  (E + r_2 A) \) and \( \mathbb{B} (E - r_2 A) \subseteq \mathbb{B} (E - r_1 A) \).

\item In particular, for any $\epsilon >0$ we have \( \mathbb{B} (E + \epsilon A) \subseteq
  \mathbb{B} (E) \subseteq \mathbb{B} (E - \epsilon A) \).

\item Therefore we have that
  $$
    \mathbb{B}_+ (E):= \bigcap_{q \in \mathbb{N}}
    \mathbb{B}(E - (1/q) A ) \qquad \text{and} \qquad
    \mathbb{B}_- (E):= \bigcup_{q \in \mathbb{N}} \mathbb{B}(E + (1/q) A ).
  $$

\item It follows that $\mathbb B_+(E)$ is closed but, even for line bundles, the
  locus $\mathbb{B}_- (E)$ is not closed in general: Lesieutre \cite{lesieutre}
  proved that this locus can be a proper dense subset of $X$, or a proper dense
  subset of a divisor of $X$.

\end{enumerate}

\end{remark}

\begin{remark}

In the case of a line bundle over the variety $X$ these loci correspond to the well-known stable base  locus $\mathbb{B} (L)$,
augmented base locus $\mathbb{B}_+ (L)$, and restricted base locus $\mathbb{B}_- (L)$.
Positivity properties of line bundles are related to asymptotic base loci as  summarized in the following table:

  \begin{center}

    \begin{tabular}{|c|c|c|c|}

      \hline

      & $\mathbb{B}_- (L)$ & $ \mathbb{B} (L) $
      & $\mathbb{B}_+ (L)$ \\
      \hline
      $=\emptyset $&  nef &  semiample & ample \\
      \hline
      $\neq X $ & pseudo-effective &  effective &  big \\

      \hline

    \end{tabular}

  \end{center}

\medskip

\end{remark}

In the case of higher rank, some of the positivity properties can be generalized in various non-equivalent ways:
one of the most natural way follows.

Let $E$ be a vector bundle on a normal  projective variety $X$, $\pi \colon
\mathbb{P}(E) \to X$ the projective bundle of rank one quotients of $E$, and
${\oo}_{\mathbb{P}(E)}(1)$ the universal quotient line bundle of $\pi^* E$ on $\mathbb{P}(E)$.
The most common way to relate positivity properties of line bundles 
to positivity properties of vector bundles is applying the definition for line bundles to
${\oo}_{\mathbb{P}(E)}(1)$:

\begin{definition}
We say that $E$ is a \emph{nef} (respectively \emph{semiample}, \emph{ample}, \emph{big}) vector bundle on X
if ${\oo}_{\mathbb{P}(E)}(1)$ is a nef (respectively semiample, ample, big) line bundle on $\mathbb{P}(E)$.

\end{definition}

\begin{remark}
We immediately have
\[
\pi (\mathbb{B} ({\oo}_{\mathbb{P}(E)}(1))) \subseteq \mathbb{B} (E) \ .
\]
More precisely, we have $\pi (\mathrm{Bs} ({\oo}_{\mathbb{P}(E)}(1))) = \mathrm{Bs} (E)$
(cf. \cite{4bros})
and we will show in the following section that the inclusion $\pi(\mathbb{B} ({\oo}_{\mathbb{P}(E)}(1)))
\subseteq \mathbb{B} (E) $ of  stable loci is strict in general.  

However in \cite{4bros} 
some useful connections are established  relying on properties
of  augmented and restricted base loci, which exhibit  a more predictable
behavior with respect to the map $\pi$.

\end{remark}

In  \cite{4bros}  we established the following:

\begin{proposition}
\label{bminplus}

Let $E$ be a vector bundle on a normal projective variety $X$,
$\pi \colon \mathbb{P}(E) \to X$ the projective bundle of one dimensional quotients of $E$,
and ${\oo}_{\mathbb{P}(E)}(1)$ the universal quotient of $\pi^* E$ on $\mathbb{P}(E)$.
Then
\begin{enumerate}
\item
$\pi(\mathbb{B}_- ({\oo}_{\mathbb{P}(E)}(1))) =  \mathbb{B}_- (E)$ ;

\item
 $ \pi(\mathbb{B}_+ ({\oo}_{\mathbb{P}(E)}(1))) =  \mathbb{B}_+ (E)$ .

\end{enumerate}

\end{proposition}

We recall some definitions appearing in the literature which can be related to the asymptotic base loci defined above:

\begin{definition}
\label{variedef}

Let $E$ be  a vector bundle on the variety $X$ .

\begin{enumerate}

\item We say that $E$ is  \emph{pseudo-effective} if  $ \mathbb{B}_- (E) \neq X$
(cf. \cite{bdpp}, where for a line bundle $L$ the locus $L_{nonnef}$ coincides with  $\mathbb{B}_- (L)$)

\item
 We say that $E$ is \emph{weakly semipositive over the (nonempty) open subset $U$} if 
 $\mathbb{B}_- (E) \subseteq X \setminus U$ (cf. \cite{vie}).

 \item We say that 
 $E$ is \emph{weakly semipositive} if it is weakly semipositive over some nonempty open subset in $X$ (cf. \cite{vie}).

\item
We say that the vector bundle $E$ is \emph{strongly semiample} if 
$\Sym^m E$ is globally generated for $m >> 0$.
(cf. \cite{demailly}).

\item We say that $E$ is \emph{asymptotically generically 
 generated} (AGG) 
if there exists a 
nonempty open subset $U \subseteq X$ and a positive iteger $m>0$ such that
$\Sym^m E$ is generated by global sections over the points in $U$.

\item
We say that $E$ is \emph{strongly big} if there exists an ample line bundle $A$,
and a positive integer $c > 0$, such that $\Sym^c E \otimes A^{-1}$ is weakly positive
(cf. \cite{jabbusch}).

\end{enumerate}

\end{definition}

It was proven in  \cite{4bros} that the following table holds for the above definitions on vector bundles:
  \begin{center}
    \begin{tabular}{|c|c|c|c|c|}
      \hline
      & $\mathbb{B}_- (E)$ & ${\overline{\mathbb{B}_- (E)}  }$  & 
$ \mathbb{B} (E) $  & $\mathbb{B}_+ (E)$ 
\vphantom{$\overline{B}^{\text{\normalsize B}}$}\\
      \hline
      $=\emptyset $& nef & nef & strongly semiample  & ample \\
      \hline
      $\neq X $ & pseudo-effective & weakly positive & AGG  &  strongly big \\
      \hline
    \end{tabular}
  \end{center}
In the following sections we further investigate the 
semiampleness and the  stable base loci of vector bundles,
and we provide some constructions which generalize Iitaka fibrations
to vector bundles, when they are strongly semiample or AGG.

\section{Semiample vector bundles}
\label{semiample}
In this section we describe and compare the 2 different notions of semi-amplitude for vector bundles existing in the literature,
and provide an Iitaka construction for strongly semiample vector bundles.

\subsection{Strongly semiample vector bundles}

In the literature two different definition for semiampleness can be found:

- either a vector bundle is called semiample when ${\oo}_{\mathbb{P}(E)}(1)$ is semiample
(one such reference can be found in \cite{proceedingskatata}, \emph{Open Problems});

- or it is called semiample when $\Sym^m E$ is globally generated for some $m > 0$
(cf. \cite{demailly}).

In order to distinguish the two definitions we call
 the latter \emph{strong semiampleness}, according to definition \ref{variedef} above.

\begin{remark}

Clearly, a vector bundle $E$ is strongly semiample if and only if $ \mathbb{B} (E) = \emptyset$; and $E$ is semiample 
if and only if $\mathbb{B} ({\oo}_{\mathbb{P}(E)}(1)) = \emptyset$.

\end{remark}

It is often believed that the two definitions coincide (cf. \cite{demailly}),
but it is indeed easy to prove that strong semiampleness implies semiampleness
(\emph{i.e.} 
$\pi (\mathbb{B} ({\oo}_{\mathbb{P}(E)}(1))) \subseteq \mathbb{B} (E) $).
However, it is not hard to construct examples where 
semiampleness holds while strong semiampleness does not, 
showing therefore that the inclusion 
$\pi (\mathbb{B} ({\oo}_{\mathbb{P}(E)}(1))) \subseteq \mathbb{B} (E) $ 
can be strict:

\begin{example}
Let $C$ be a smooth projective curve of genus $g \geqslant 1$. And let $E = \oo_C \oplus L $ with $L$ being a nontrivial torsion line bundle
satisfying $L^{\otimes 2} = \oo_C$. Clearly $L$ is not globally generated as it has no nontrivial global sections.
As in $\Sym^c E$ there is a direct factor isomorphic to $L$, then for any $c>0$ the vector bundle 
$\Sym^c E$ cannot being generated by global sections in any point.

However it can be proven that  $\oo_{\pp(E)} (1)$ is semiample: it is easily proven by hand that  $\oo_{\pp(E)}(2)$ is globally generated on $\pp(E)$,
but we remark that we can apply Catanese-Dettweiler criterion for semiampleness (cf. \cite{catdett}), 
as $E$ is a unitary flat (poly-stable of degree $0$ over a curve) vector bundle on the curve $C$,
and clearly the image of its monodromy representation is a finite group of order $2$.

\end{example}

The example above also shows that strong semiampleness is not stable with respect to finite morphisms,
contrarily to what happens to semiampleness.

\subsection{Semiample ``fibration''}

Let $E$ be a rank $r$ strongly semiample vector bundle on a normal complex projective variety $X$. 
Set $\mathbf{M} (E) := \{ m \in \mathbb{N} ~|~  S^m E  \textrm{ is globally generated} \}$.
Denote
\[
\sigma_m (r) := \dim S^m E = \binom{m+r-1}{m}
\]
then for any $m \in \mathbf{M} (E)$ evaluation on global sections of $S^m E$ yields a morphism:
\[
\begin{array}{rcl}
\varphi_m \colon  X & \to  & \G (H^0 (X,S^m E) , \sigma_m(r))\\
x & \mapsto  & [H^0(X, S^m E) \twoheadrightarrow S^m E (x)]
\end{array}
\]
where $\G (H^0 (S^m E) , \sigma_m(r))$ is the grassmannian of $\sigma_m (r)$-dimensional quotients.
We call these morphisms \emph{Kodaira maps} (cf. \cite{BCL}).

Ampleness and some other positivity properties cannot be detected through the geometry of Kodaira maps
(cf. \cite{PAG} 6.1.6),
however it is interesting to observe the asymptotical behavior of these Kodaira maps $\varphi_m$: in the case of 
(semiample) line bundles, these maps factor through a tower of finite maps, so that their images all have the same dimensions,
and for large $m$ one gets a given fibration, which is the Stein factorization of any of the maps $\varphi_m$.

Something similar happens for higher rank vector bundles,
but we do not get a fibration in general. The purpose of this section is to investigate this asymptotic construction.

\begin{lem}
\label{lemmafinite}

Let $E$ be a  strongly semiample vector bundle on a normal projective variety $X$, let 
$\varphi_m$ be the Kodaira maps described above, for any  integer $m \in \mathbf{M} (E) $, 
and call $Y_m$ the normalization of the image of the map $\varphi_m $.
Then for any  positive integer ${k}$, the integer $km$ lies in $\mathbf{M} (E) $ and the map
$\varphi_m$ factors through $\varphi_{km}$ and a finite morphism $Y_{km} \to Y_m$. 

\end{lem}

In particular this shows that all the images $Y_m$ have the same dimension.

\begin{proof}

Obviously if  $S^m E$ is globally generated  so is $S^{km} E$. 

We want to compare the maps 
$\varphi_m \colon  X  \to  Y_m \to \G (H^0 (X,S^m E) , \sigma_m(r)) $
and $\varphi_{km} \colon  X  \to  Y_{km} \to \G (H^0 (X,S^{km} E) , \sigma_{km}(r)) $,
by factoring the first one through a finite map $Y_{km} \to Y_m$.

Let us call $\F_{k,m}$ the flag variety of quotients
$S^k H^0 (X,S^{m} E) \twoheadrightarrow Q_1 \twoheadrightarrow Q_2$, with 
$\dim Q_1 = \sigma_k (\sigma_m (r))$ and $\dim Q_2
 = \sigma_{km} (r)$, and $\psi_{k,m} \colon X \to \F_{k,m}$ the map
 $x \mapsto [S^k H^0 (X,S^{m} E) \twoheadrightarrow S^k (S^m E(x)) \twoheadrightarrow S^{km} E (x)]$.
Let us call $\widehat{Y}_{k,m} $  the normalization of the  image of $X$ in $\F_{k,m}$,  $\pi_1$ and $\pi_2$ the natural projections
from the flag variety to the 2 Grassmannians.
 
 Finally, after calling $W$ the image of the natural linear map 
 \[
 S^k H^0 (X, S^m E) \to H^0(X, S^{km}E) ~,
 \]
 we can construct the following diagram:

\begin{center}
\begin{tikzpicture}
\node (A) at (-2, 0){$X$};
\node (B) at (1,0) {$\widehat{Y}_{k,m}$};
\node (C) at (4,0){$\F_{k,m}$};
\node (D) at (3,1){$\widetilde{Y_2}$};
\node (E) at (3,2){$\widetilde{Y_2}$};
\node (F) at (3,3){$Y_{km}$};
\node (G) at (3,-1){$\widetilde{Y_1}$};
\node (H) at (3,-2){$Y_m$};
\node (I) at (6.5,1){$ \G (S^k H^0 (X,S^{m} E) , \sigma_{km}(r))$};
\node (J) at (6.5,2){$ \G (W, \sigma_{km}(r)) $};
\node (K) at (6.5,3){$ \G (H^0 (X,S^{km} E) , \sigma_{km}(r)) $};
\node (L) at (6.5,-1){$ \G (S^k H^0 (X,S^{m} E) , \sigma_{k}(\sigma_m(r)))$};
\node (M) at (6.5,-2){$ \G (H^0 (X,S^{m} E) , \sigma_{m}(r))$};
\path[->>] (A) edge node[above]{$\psi_{k,m}$}(B);
\path[->] (B) edge node[right]{}(C);
\path [->>] (B) edge node[above,rotate=30]{$\sim$} (D);
\path[->>] (B) edge node[below]{$\mu$}(G);
\path[->] (H) edge node[below,rotate=90]{$\sim$}(G);
\path[->>] (A) edge node[above]{$\varphi_{km}$~~~~}(F);
\path[->>] (A) edge node[below]{$\varphi_m$}(H);
\path[->>] (F) edge node[right]{$\nu$}(E);
\path[->] (E) edge node[below,rotate=90] {$\sim$} (D);
\path[->] (F) edge node[above]{}(K);
\path[->] (E) edge node[above]{}(J);
\path[->] (D) edge node[above]{}(I);
\path[->>] (C) edge node[below]{$\pi_2$}(I);
\path[->>] (C) edge node[above]{$\pi_1$}(L);
\path[->] (G) edge node[above]{}(L);
\path[->] (H) edge node[above]{}(M);
\path[dotted,->>] (K) edge node[above]{}(J);
\path[right hook->] (J) edge node[above]{}(I);
\path[right hook->] (M) edge node[above]{}(L);
\end{tikzpicture}
\end{center}
where the maps between grassmannians are constructed as follows:
\begin{itemize}
\item the first map above is a (rational) projection induced on the grassmannians (of quotients)
by the inclusion $W \subseteq H^0(X, S^{km}E)$;
\item the second map is induced by the quotient $S^k H^0(X, S^m E) \twoheadrightarrow W$;
\item the last map below is induced by considering the $k$-th 
symmetric product of a quotient $H^0(X, S^m E) \twoheadrightarrow Q$.
\end{itemize}

As evaluating global sections on a point $x \in X$ commutes with the various symmetric products considered,
it is easy to see that the maps above are well defined and make all diagrams commute.
As an example the map $\pi_2 \circ \psi_{k,m}$ is obtained on $x$ by evaluating
$S^k H^0(X, S^m E) \twoheadrightarrow S^{km} E (x)$, as this maps factors through
$S^k H^0(X, S^m E) \twoheadrightarrow W \hookrightarrow H^0(X, S^{km} E)$ then we see that
$\pi_2 \circ \psi_{k,m} = \nu \circ \varphi_{km}$.

In order to complete the proof of the lemma we have to show that:
\begin{enumerate}
\item the map $\widehat{Y}_{k,m} \twoheadrightarrow \widetilde{Y_2}$ is an isomorphism;
\item the map $\nu \colon Y_{km} \twoheadrightarrow \widetilde{Y_2}$ is finite;
\item the map $\mu \colon \widehat{Y}_{k,m} \twoheadrightarrow \widetilde{Y_1}$ is finite.
\end{enumerate}

The first point follows observing that if the evaluation map
\[
S^k H^0 (X,S^{m} E) \twoheadrightarrow S^{km}E(x)
\]
 coincide on two points $x_1, x_2 \in X$ (\emph{i.e.} their kernels are equal) then also 
the images of $x_1$ and $x_2$ in $\F_{k,m}$ coincide:
suppose by contradiction that the evaluation maps
$S^k H^0 (X,S^{m} E) \twoheadrightarrow S^{km}E(x)$
 coincide on $x_1$ and $x_2$, but the maps
$S^k H^0 (X,S^{m} E) \twoheadrightarrow S^k (S^m E(x)) $
do not coincide on those points.
Then there exists a section $s \in H^0(X, S^m E)$ such that
$s(x_1)=0 \in S^m(E(x_1))$ and $s(x_2) \neq 0 \in S^m(E(x_2))$,
so we have $s(x_1)^k=0 \in S^{km}(E(x_1))$ and $s(x_2)^k \neq 0 \in S^{km}(E(x_2))$ as well,
in contradiction with the assumption.

Then the projection $\pi_2$ from the flag variety {$\F_{k,m}$} to the Grassmannian 
$ \G (S^k H^0 (X,S^{m} E) , \sigma_{km}(r))$
induces a bijection on the images from $X$, and therefore an isomorphism 
on their normalizations.

The second point follows
 observing that $\nu$ is a regular map on $Y_{km}$ which is the restriction 
of an affine map (the projection) to a proper variety, hence is finite.

The last point can be proven  by contradiction: suppose that  an irreducible curve $C \subseteq \widehat{Y}_{k,m}$ 
is contracted by $\mu$,
then there exists an irreducible curve $\widetilde{C}$ in $X$ dominating $C$ which is contracted by 
$\varphi_m$.
Therefore $S^m E _{|\widetilde{C}} \cong \oo_{\widetilde{C}}^{\oplus\sigma_m(r)}$ is trivial
and so are $S^k (S^m E)_{|\widetilde{C}}$ and $S^{km} E_{|\widetilde{C}}$,
so that evaluation on global sections is constant on points of $\widetilde{C}$,
\emph{i.e.} the map $\psi_{k,m}$ contracts $\widetilde{C}$, which contradicts our hypothesis.

\end{proof}

The above Lemma implies  the following

\begin{theorem}[Iitaka for strongly semiample vector bundles]
\label{itakasemiample}

Let $X$ be a normal projective variety, $E$ a vector bundle as above, 
$\varphi_m$ the Kodaira maps and $Y_m$ the normalizations of their images. 
Then there exists a diagram 
\begin{center}
\begin{tikzpicture}
\node (A) at (-1, 0){$X$};
\node (B) at (2,1) {$Y_{\infty}$};
\node (C) at (2,-1){$Y_{\G}$};
\path[->] (A) edge node[above]{$\varphi_{_{{\rm det}E}}$~~~}(B);
\path[->] (B) edge node[right]{$\pi$}(C);
\path[->] (A) edge node[below]{$\varphi_{\G}$}(C);
\end{tikzpicture}
\end{center}
where $\varphi_{_{{\rm det}E}}$ is the semiample fibration (Iitaka) induced by the determinant line bundle of $E$ and
$\pi$ is a finite map,
such that  for any $m \in \mathbf{M}(E)$
and  for $k \gg 0$ we have $Y_{km} =  Y_{\G}$ and $\varphi_{km} = \varphi_{\G}$.
In particular, any Kodaira map $\varphi_m$, for $m \in \mathbf{M}(E)$,  factors through
$\varphi_{\G}$ and a finite map.

Furthermore, for all $m \in \mathbf{M}(E)$
there exist vector bundles $Q_m$ on $Y_{\G}$ such that 
$\varphi^* Q_m = S^m E$.

\end{theorem}

\begin{proof}
Applying Lemma \ref{lemmafinite} above we see that all $Y_m$ are dominated by $Y_{km}$ and that the corresponding Kodaira maps factor through finite maps,
hence there must be an inverse limit of all such maps , corresponding to $Y_{km}$ for $k \gg0$, call such limit 
$Y_{\G}$.

Now we have $Y_{\G} \subseteq \G(H^0(X, S^n E), \sigma_n (r)) \subseteq \mathbb{P} (\bigwedge^{\sigma_n(r)}H^0(X, S^n E))$ for some $n \in \mathbf{M} (E)$,
where the latter is Pl{\" u}cker embedding.
So the map $\varphi_{\G}$ is determined by the  linear subseries $W \subseteq H^0(X , (\det E)^{\otimes N})$,
where $(\det E)^{\otimes N} = \det (S^n E)$, and $W = \mathrm{Im}(\bigwedge^{\sigma_n(r)}H^0(X, S^n E) \to H^0(X , (\det E)^{\otimes N}))$ .

Then $\varphi_{\G}$ factors through 
$\varphi_{_{{\rm det}E}} \colon X \to  \mathbb{P} (H^0(X, (\det E)^{\otimes N}))$ and a projection to $\mathbb{P} (W)$,
the first map being exactly the Iitaka fibration for the line bundle $\det E$, and the projection being a finite map for the same reason as the map 
$\nu$ in the proof of the  above lemma.

The last property stated in the theorem follows pulling back to $Y_{\G}$, for each $m \in \mathbf{M}(E)$, 
the canonical quotients of the grassmannians 
through the maps
\[
Y_{\G} \to Y_m \to \G (H^0(X, S^m E), \sigma_m(r)) ~.
\]
\end{proof}

\begin{remark}
\label{notfib1}

The finite maps $Y_{km} \to Y_m$ induce finite extensions of the fields 
of rational functions
$k(Y_m) \subseteq k(Y_{km})$. These extensions are all included in the field
of rational functions of $X$, in particular they are included in 
the algebraic closure of $k(Y_m)$ in $k(X)$. For $m\gg 0$ they stabilize
to $k(Y_{\mathbb{G}})$. As $\varphi_{\det E}$ is a fibration, 
then $k(Y_{\infty})$ is algebraically closed in $k(X)$, 
and it is in fact the algebraic closure of $k(Y_m)$ in $k(X)$.
However the field $k(Y_{\mathbb{G}})$ needs not to be algebraically closed 
in $k(X)$. 
This corresponds to the fact that in general the map $\pi$ 
appearing in the theorem is not an isomorphism but just a finite map,
as it is shown in the example below.
\end{remark}

\begin{remark}
\label{notfib2}

One can wonder whether the map induced by global sections of the symmetric power of a vector bundle 
is actually nothing more but the Iitaka fibration of the determinant bundle, seen after Pl{\" u}ker embedding of the Grassmannian.
Actually this is not always the case, so that the finite map $\pi$ appearing in the theorem is not always an isomorphism,
and its degree is therefore an invariant of the vector bundle $E$. An example where the map $\pi$ is not an isomorphism is given below.
\end{remark}

\begin{example}
\label{counterexdouble}

Fix a vector space $V$ of dimension $3$, and let $\pi \colon X \to  \mathbb{P}(V) = \mathbb{P}^2$ be a double cover, ramified over a smooth conic of $\mathbb{P}^2$, 
\emph{i.e.} $X \cong \mathbb{P}^1 \times \mathbb{P}^1$.
Consider the Euler exact sequence on $\mathbb{P}^2$:
\[
0 \to \oo_{\mathbb{P}^2}(-1) \to V^* \otimes \oo_{\mathbb{P}^2} \to \mathcal{Q} \to 0
\]
call $E: = {\pi}^* \mathcal{Q} $, then we can prove that for all powers $S^m E$ the Kodaira maps are finite of degree $2$.
In fact, ${\pi}_* \oo_X =  \oo_{\mathbb{P}^2} \oplus \oo_{\mathbb{P}^2} (-1)$,
and by projection formula we have:
\[
H^0(X, S^m E) = H^0(\mathbb{P}^2, \pi_* S^m E)= H^0(\mathbb{P}^2, S^m \mathcal{Q}) \oplus  H^0(\mathbb{P}^2, S^m \mathcal{Q}(-1)) ~,
\]
and since  for all $m>0$ we have that  $H^0(\mathbb{P}^2, S^m \mathcal{Q}(-1)) =0$, then 
all global sections of $S^m E$ over $X$ are pull back of sections of $S^m \mathcal{Q}$ over ${\mathbb{P}^2}$,
so the Kodaira maps factor through the degree $2$ map $\pi$.

\end{example}

\section{Iitaka-type properties for vector bundles}
\label{iitaka}
We now study the Iitaka construction in case of non strongly semiample vector bundles.
In order to have (rational) maps to Grassmannian varieties we need some generation property,
which is a rather strong positivity property, that corresponds to effectivity for line bundles.

\begin{definition}
We recall the definition of \emph{asymptotic generic generation} given above in section \ref{baseloci}.

\begin{enumerate}

\item Let $X$ be a normal projective complex variety. Let $E$ be a vector bundle on $X$, with ${\rm rk}(E)=r$,
then $E$ is said to be \emph{asymptotically generically generated} (AGG), if global sections of some symmetric power $S^m E$
generate $S^mE$ over an open dense subset of $X$.

\item Let $E$ be an AGG vector bundle over a variety $X$. Then 
denote $\mathbf{N} (E) := \{ m \in \mathbb{N} ~|~  Bs( S^m E) \neq X  \}$

\end{enumerate}

\end{definition}

\begin{remark}

A vector bundle $E$ is asymptotically generically generated iff 
for some $m>0$ the evaluation map $H^0(X, S^m E) \otimes \oo_X  \to S^m E$ is surjective over an open subset of $X$,
iff 
$\mathbb{B} (E) \neq X$.

\end{remark}

For asymptotically generically generated vector bundles,
it makes sense to consider the Kodaira maps,
which in this case are rational maps
\[
\varphi_m \colon X \dashrightarrow \G(H^0(X, S^m E), \sigma_m(r))
\]
and study their asymptotic behavior.
For $m\gg 0$ in $\mathbf{N} (E)$  these maps are regular out of the stable base locus
$\mathbb{B} (E)$, call $Y_m$ the images.
We will construct a diagram similar to the one appearing in 
Lemma \ref{lemmafinite}. Note that we can consider the images of Kodaira maps without normalizing,
as in this case we are only interested in the birational behavior.

\begin{center}
\begin{tikzpicture}
\node (A) at (-2, 0){$X \setminus \mathbb{B}(E)$};
\node (B) at (1,0) {$\widehat{Y}_{k,m}$};
\node (C) at (4,0){$\F_{k,m}$};
\node (D) at (3,1){$\widetilde{Y_2}$};
\node (E) at (3,2){$\widetilde{Y_2}$};
\node (F) at (3,3){$Y_{km}$};
\node (G) at (3,-1){$\widetilde{Y_1}$};
\node (H) at (3,-2){$Y_m$};
\node (I) at (6.5,1){$ \G (S^k H^0 (X,S^{m} E) , \sigma_{km}(r))$};
\node (J) at (6.5,2){$ \G (W, \sigma_{km}(r)) $};
\node (K) at (6.5,3){$ \G (H^0 (X,S^{km} E) , \sigma_{km}(r)) $};
\node (L) at (6.5,-1){$ \G (S^k H^0 (X,S^{m} E) , \sigma_{k}(\sigma_m(r)))$};
\node (M) at (6.5,-2){$ \G (H^0 (X,S^{m} E) , \sigma_{m}(r))$};
\path[->>] (A) edge node[above]{$\psi_{k,m}$}(B);
\path[right hook->] (B) edge node[right]{}(C);
\path [->>] (B) edge node[above,rotate=30]{} (D);
\path[->>] (B) edge node[below]{$\mu$}(G);
\path[->] (H) edge node[below,rotate=90]{$\sim$}(G);
\path[->>] (A) edge node[above]{$\varphi_{km}$~~~~}(F);
\path[->>] (A) edge node[below]{$\varphi_m$}(H);
\path[->>] (F) edge node[right]{$\nu$}(E);
\path[->] (E) edge node[below,rotate=90] {$\sim$} (D);
\path[right hook->] (F) edge node[above]{}(K);
\path[right hook->] (E) edge node[above]{}(J);
\path[right hook->] (D) edge node[above]{}(I);
\path[->>] (C) edge node[below]{$\pi_2$}(I);
\path[->>] (C) edge node[above]{$\pi_1$}(L);
\path[right hook->] (G) edge node[above]{}(L);
\path[right hook->] (H) edge node[above]{}(M);
\path[dotted,->>] (K) edge node[above]{}(J);
\path[right hook->] (J) edge node[above]{}(I);
\path[right hook->] (M) edge node[above]{}(L);
\end{tikzpicture}
\end{center}

\begin{remark} In the latter diagram the maps $\mu$ and $\nu$ are not necessarily finite maps.
\end{remark}

This construction yields the following

\begin{theorem}[Iitaka for AGG vector bundles]
Let $X$ be a projective variety and $E$ be an asymptotically generically generated bundle over $X$.
Call $\varphi_m \colon X \dashrightarrow Y_m$ the evaluation rational maps defined as above.
Then there exist projective varieties $X_{\mathbb{G}}$ and
$Y_{\mathbb{G}}$ together with regular surjective morphisms $u_{\mathbb{G}} \colon X_{\mathbb{G}} \to X$ and 
${\varphi}_{\mathbb{G}} \colon X_{\mathbb{G}} \to Y_{\mathbb{G}}$ such that 
for every $m \in \mathbf{N} (E)$ and for $k \gg 0$ the regular map ${\varphi}_{\mathbb{G}} \colon X_{\mathbb{G}} \to Y_{\mathbb{G}}$ is a birational model for 
the rational map $\varphi_{km}$, \emph{i.e.} we have the following commutative diagram
\begin{center}
\begin{tikzpicture}
\node (A) at (0,1){$X$};
\node (B) at (0,0) {$Y_{km}$};
\node (C) at (2,1){$X_{\G}$};
\node (D) at (2,0){$Y_{\G}$};
\path[->,dashed] (A) edge node[left]{$\varphi_{km}$}(B);
\path[->,dashed] (D) edge node[below]{$\nu_{km}$}(B);
\path[->] (C) edge node[right]{$\varphi_{\G}$}(D);
\path[->] (C) edge node[above]{$u_{\G}$}(A);
\end{tikzpicture}
\end{center}
where the horizontal maps are birational.
\end{theorem}

\begin{proof}

Applying the construction above we have for each $m$ 
\[
k(X) \supseteq k(Y_{km}) \supseteq k(Y_m)
\]
so that all the $Y_m$'s are dominated by a common limit which is birational to $Y_{km}$ for $k$ big enough.
Fixing big enough $m$ and $k$, we can choose a compactification of $Y_{km}$ 
to be $Y_{\G}$ and resolving the indeterminacies we have a model 
${\varphi}_{\mathbb{G}} \colon X_{\mathbb{G}} \to Y_{\mathbb{G}}$
which is birational to all the other ones.
\end{proof}

\begin{remark}
We already noticed that,
contrary to the line bundle case,
the map ${\varphi}_{\mathbb{G}} \colon X_{\mathbb{G}} \to Y_{\mathbb{G}}$ is not a fibration in general
 (cf. Remarks \ref{notfib1} and \ref{notfib2}),
however one can consider the Stein factorization 
$X_{\mathbb{G}} \to Y_{\infty} \to  Y_{\mathbb{G}}$
for this map.
In the case of a strongly semiample vector bundle $E$ over $X$
we have that $X_{\mathbb{G}} = X$ and observed that the map $X_{\mathbb{G}} \to Y_{\infty} $
is exactly the Iitaka fibration for the line bundle 
$\det E$ over $X$, therefore in the strongly semiample case
the dimension of $Y_{\G}$ is the Iitaka dimension of the determinant of $E$.

 This needs not be the case for non semiample vector bundles,
so it makes sense to ask the following questions.

\end{remark}

In the following assume $E$ to be  AGG:

\begin{question}
\label{kodim}

Is it always $\dim Y_{\G} = k(X, \det E)$?

\end{question}

\begin{question}

Suppose $E$ to be strongly big but not strongly semiample, is $\dim Y_{\G} = \dim X$?

\end{question}

\begin{question}

Suppose $E$ is big (and AGG), what can we say about $\dim Y_{\G}$?

\end{question}

\begin{question}

Can we relate the complement of the augmented base loci to the Kodaira maps as it happens for line bundles
(cf. \cite{BCL})?
\end{question}

\begin{remark}
The last question has a negative answer if we look for a direct generalization of the theorem
in \cite{BCL}:  among other results the authors prove that the augmented base locus of a line bundle is the complementary of the biggest open subset 
where the Iitaka fibration is an isomorphism. However the example \ref{counterexdouble} shows that this cannot be generalized to vector bundles,
even if we were to change the word \emph{isomorphism} by \emph{finite regular map}.
In fact if we consider the tautological quotient $\mathcal{Q}$ of rank $2$ over $\pp^2 = \pp(V)$,
then its Kodaira map $\varphi_1$ is the isomorphism $\pp(V) \cong \mathbb{G}r(V^*, 2)$,
so all the Kodaira maps $\varphi_m$ are isomorphisms,
however $\mathbb{B}_+ (\mathcal{Q}) = \pp^2$.

\end{remark}

In any case  we can define some \emph{asymptotic invariants} for an AGG vector bundle:

\begin{definition}
Let $E$ be an AGG vector bundle, using the notations above we call
\begin{enumerate}

\item \emph{Iitaka index} of $E$ the integer ${FI}(E) = \deg (Y_{\infty} \to Y_{\G})$;

\item \emph{Iitaka dimension} of $E$ the integer $k(X, E) = \dim Y_{\G}$; 

\end{enumerate}

\end{definition}

%

Finally, we remark that in \cite{fujiwara} the author gives a characterization for varieties with 
semiample cotangent bundles and Kodaira dimension 0 or 1. 
We can prove the following theorem about varieties with strongly semiample cotangent bundle and Kodaira dimension 0: 

\begin{theorem}
\label{abeliancar}
Let $X$ be a smooth projective variety, and $\Omega^1_X$ its cotangent bundle. The following conditions are equivalent:
\begin{enumerate}
\item $X$ is an abelian variety;
\item \label{komega} the cotangent bundle $\Omega_X$ is strongly semiample and the Iitaka dimension of $\Omega_X$ vanishes: $\mathbb{B}(\Omega^1_X)= \emptyset$ and $k(X, \Omega_X)=0$;
\item \label{kkodaira} the cotangent bundle $\Omega_X$ is strongly semiample and the Kodaira dimension of $X$ vanishes: $\mathbb{B}(\Omega^1_X)= \emptyset$ and $k(X, K_X)=0$;
\item \label{ktrivial} $\Sym^m \Omega^1_X$ is trivial for some $m>0$.
 \end{enumerate}
\end{theorem}

\begin{proof}

$\mathbb{B}(\Omega^1_X)= \emptyset$ if and only if $\Omega^1_X$ is strongly semiample,
and in this case by Theorem
\ref{itakasemiample}
$k(X, \Omega^1_X)= k(X, K_X)$, then (\ref{komega}) and (\ref{kkodaira}) are equivalent.
Since $k(X, \Omega^1_X)=0$ then some symmetric power of the cotangent bundle is the pull-back from a point, 
hence it is trivial,
it follows that (\ref{komega}) and (\ref{kkodaira}) are equivalent to point (\ref{ktrivial}).

(\ref{komega}), (\ref{kkodaira}) and (\ref{ktrivial}) are  necessarily satisfied if $X$ is an abelian variety,
 let us show that they are sufficient as well.

Let us call $d := \dim X$ and suppose  $\Sym^m \Omega^1_X \cong \oo^{\oplus \sigma_m(d)}_X$ for some $m>0$.
Then by \cite[Theorem I]{fujiwara}
there is a finite \'etale Galois cover $f \colon A \to X$ where $A$ is an abelian variety.
Let us denote $G$ the finite group acting freely on $A$ such that $X = A/G$.
We will prove that within our hypothesis $G$ acts by translations on $A$
and therefore  $X$ is an abelian variety.

We have $f^* \Omega^1_X \cong \Omega^1_A \cong \oo^{\oplus d}_A$, 
 let $V: = H^0 (A , \Omega^1_A)$.
Then $G$ acts on $V$ via its action on $A$, let $\rho \colon G \to GL(V)$ be this action.
If the action $\rho$ is trivial, then $G$ acts on $A$ by translations.
In fact, since $\Sym^m \Omega^1_X $ is trivial, then 
$f^* H^0(X, \Sym^m \Omega^1_X) \cong H^0(A, \Sym^m \Omega^1_A) = \Sym^m V$,
therefore $\Sym^m (\rho)$ is the trivial action on $\Sym^m V$. 
This is easily seen to imply that $G$ acts by homotheties on $V$,
in fact the action of each element $g \in  G$ is diagonalizable as $G$ is a finite group,
and using the triviality of the  action on the symmetric product it is not difficult to show that all
the eigenvalues must coincide.
Therefore each $g \in  G$ acts on $V$ by $g \cdot v = \lambda_g v$, and on  $A$ by 
$g \cdot x = \lambda_g x +\tau$. Then if $\lambda_g \neq 1$ there is a fixed point in $A$, and 
this cannot be as the quotient is \'etale by hypothesis. So $G$ acts trivially on $V$ and acts by translations on $A$,
therefore $X$ is an abelian variety. Notice that in this case all holomorphic 1-forms on $V$ descend to $X$, 
and in fact \emph{a posteriori} $\Omega^1_X$ is globally generated.
\end{proof}

\begin{qst}

We can ask whether the above theorem extends to a birational criterion:
\begin{enumerate}
\item Is it true that $\mathbb{B}(\Omega^1_X) \neq X$ and Kodaira dimension $k(X, K_X)=0$ implies that $X$ is birational to an abelian variety?
\item Is it true that $\mathbb{B}(\Omega^1_X) \neq X$ and $k(X, \Omega^1_X)=0$ implies that $X$ is birational to an abelian variety?
\end{enumerate}
\end{qst}

As for the first question we can remark that it can be reduced to smooth minimal models of $X$ (when they exist):

\begin{lemma}
\label{extend1forms}
Let $X$ and $Y$ be two birational smooth projective varieties. 
Then $\mathbb{B}(\Omega^1_X) \neq X$ if and only if $\mathbb{B}(\Omega^1_Y) \neq Y$,
\emph{i.e.} $X$ has AGG cotangent bundle if and only if $Y$ has AGG cotangent bundle.
\end{lemma}

\begin{proof}

Suppose $f \colon X \to Y$ is a regular and birational morphism between smooth varieties $X$ and $Y$, 
with center $Z \subset Y$ and exceptional divisor $E \subset X$. We have an exact sequence on $X$:
\[
0 \to f^* \Omega^1_Y \to \Omega^1_X \to \mathcal{F} \to 0
\]
where $\mathcal{F}$ is a sheaf supported on the exceptional divisor $E$ of $f$.

Now if $\mathbb{B}(\Omega^1_Y) \neq Y$ then some symmetric product $S^m \Omega^1_Y$ is generated by global sections over an open subset
of $Y$, so by taking the pull-back of the symmetric product of 1-forms we have that $S^m \Omega^1_X$ is generated by global sections over an open subset as well, so $\mathbb{B}(\Omega^1_X) \neq X$.

Vice-versa, suppose that $\mathbb{B}(\Omega^1_X) \neq X$: as $\codim_Y Z \geqslant 2$, we can restrict 1-forms on $X$ to 1-forms on 
$X \setminus E \cong Y \setminus Z$, and then extend them to 1-forms on $Y$,
so if $S^m \Omega^1_X$ is generated by global sections over an open subset the same happens for $S^m \Omega^1_Y$,
the two vector bundles  being isomorphic over $X \setminus E \cong Y \setminus Z$.

\end{proof}

The answer to the first question is positive in case $\dim X \leqslant 2$:

\begin{theorem}

Let $X$ be a smooth projective variety such that $\dim X \leqslant 2$ and $k(X, K_X) = 0$.
Then $X$ is birational to an abelian variety if and only if $\mathbb{B}(\Omega^1_X) \neq X$.

\end{theorem}

\begin{proof}
It is obvious in dimension 1.
Let us observe first that if $X$ is birational to an abelian variety then the Albanese morphism is a birational map,
and the pull-back of holomorphic 1-forms from the Albanese variety to $X$ gives $\mathbb{B}(\Omega^1_X) \neq X$.

Let us prove then that Kodaira dimension $k(X, K_X) = 0$ and $\mathbb{B}(\Omega^1_X) \neq X$ imply birationality to an abelian surface.

By Lemma \ref{extend1forms} we can  suppose  that $X$ is a minimal surface. Then it is an abelian surface, or a bielliptic surface, or a $K3$ or an Enriques surface.
Now a bielliptic surface is a smooth quotient of an abelian surface, so its cotangent bundle cannot be generically generated by global sections for the same argument as in Theorem \ref{abeliancar}, and the same applies for its symmetric powers.
If $X$ is a $K3$ surface, then the symmetric powers of the cotangent bundle have no global sections, as it is proven in Theorem 7.8 in \cite{bdpp}, so the same happens on an Enriques surfaces as they are quotients of $K3$'s.
\end{proof}

Finally, we notice that in the work of the first author \cite{stabtrans} some rational maps are constructed from Grassmannians to  moduli spaces of vector bundles over a curve,
and it would be interesting to see what kind of constructions could lead considering Kodaira maps from these moduli spaces.



\bibliographystyle{amsalpha}
\bibliography{iiitaka}

\end{document}